\documentclass[leqno,12pt]{amsart}
\usepackage[top=30truemm,bottom=30truemm,left=25truemm,right=25truemm]{geometry}
\usepackage{amssymb}
\usepackage{amsmath}
\usepackage{amsthm}
\usepackage{amscd}
\usepackage{mathrsfs}
\usepackage{graphicx}
\usepackage{color}
\usepackage{url}
\usepackage{enumitem}
\makeatletter

\@addtoreset{equation}{section}
\@namedef{subjclassname@2020}{\textup{2020} Mathematics Subject Classification}
\makeatother
\theoremstyle{plain}
\newtheorem{theorem}{\indent\bf Theorem}[section]

\newtheorem{corollary}[theorem]{\indent\bf Corollary}

\theoremstyle{definition}

\newtheorem{example}[theorem]{\indent\bf Example}

\newcommand{\ai}{\sqrt{-1}}
\newcommand{\tr}{\mathrm{tr}}
\newcommand{\R}{\mathbb{R}}
\newcommand{\C}{\mathbb{C}}

\newcommand{\dV}{\mathrm{d}V}

\newcommand{\dx}{\mathrm{d}x}
\newcommand{\Levi}{\mathcal{L}}

\begin{document}
\pagestyle{plain}
\thispagestyle{plain}

\title[Partial positivity is not preserved by marginalization]
{Partial positivity is not preserved by marginalization}

\author[T. INAYAMA]{Takahiro INAYAMA}
\address{Department of Mathematics\\
	College of Science\\
	Rikkyo University\newline
	3-34-1, Nishi-Ikebukuro, Toshima-ku\\
	Tokyo, 171-8501\\
	Japan
}
\email{5073974@rikkyo.ac.jp}
\email{inayama570@gmail.com}

\begin{abstract}
	In this paper, we show that partial positivity is not preserved by marginalization in both the real and complex settings.
	We also give a sufficient condition for the preservation of partial positivity under marginalization.
\end{abstract}


\maketitle
\setcounter{tocdepth}{2}

\section{Introduction}\label{sec:intro}

The classical Pr\'ekopa theorem states that marginalization preserves convexity.
More precisely, if $X\subset\R^r$ is convex and $\psi$ is convex on $X_t\times\R^n_x$, then
\[
	\widetilde{\psi}: t\longmapsto-\log\int_{\R^n}e^{-\psi(t,x)}\,\dx
\]
is convex whenever the integral is finite \cite{Pre73}.

Berndtsson proved a complex analogue \cite{Ber98}.
Let $U\subset\C^r_\tau$ be a domain and let $D\subset\C^n_z$ be a pseudoconvex Reinhardt domain.
If $\varphi$ is plurisubharmonic on $U\times D$ and independent of $\arg(z_j)$ for $j=1,\ldots,n$, then
\[
	\widetilde{\varphi}: \tau\longmapsto-\log\int_D e^{-\varphi(\tau,z)}\,\dV(z)
\]
is plurisubharmonic.

The classical notion of $q$-positivity originates in Andreotti--Grauert theory \cite{AG62}.
A smooth function on a Hermitian manifold of dimension $n$ is $q$-positive if its Levi form has at least $n-q$ positive eigenvalues, and is uniformly $q$-positive if the sum of any $q+1$ Levi eigenvalues, counted with multiplicity, is positive \cite[Definition~2.1]{Yan19}.

In \cite[Problem~3]{Ina21}, the author asked whether an analogous conclusion holds for uniform $(q-1)$-positivity.
Namely, if the sum of any $q$ eigenvalues of $\Levi_{\varphi}$, counted with multiplicity, is positive, does it follow that the sum of any $q$ eigenvalues of $\Levi_{\widetilde\varphi}$ is also positive?
Here for a smooth real-valued function $\eta$ on a domain in
$\C^N$ with standard coordinates $w=(w_1,\ldots,w_N)$, we write
\[
	\Levi_\eta
	:=\left(
	\frac{\partial^2\eta}{\partial w_j\partial\overline w_k}
	\right)_{j,k=1}^N
\]
for its Levi matrix.
This problem has remained open for several years.
In this paper, we construct a counterexample to this question.

\begin{theorem}\label{mainthm:counterexample}
	Fix $r\geq q\geq2$ and $n\geq1$.
	There exists a family of smooth real-valued functions $\varphi_K$ $(K\geq 1)$ on the unit polydisc $\Delta^r\times\Delta^n$ such that $\varphi_K$ is independent of $\arg(z_j)$ for $j=1,\ldots,n$ and the sum of the $q$ smallest eigenvalues of $\Levi_{\varphi_K}$ is greater than $1/5$, whereas the corresponding sum for $\Levi_{\widetilde\varphi_K}(0)$ tends to $-\infty$ as $K\to +\infty$.
\end{theorem}

Not only does this theorem imply that uniform $(q-1)$-positivity is not preserved by marginalization, but it also shows that the sum of the $q$ smallest eigenvalues of $\Levi_{\widetilde\varphi_K}$ can be arbitrarily negative, even if the sum of the $q$ smallest eigenvalues of $\Levi_{\varphi_K}$ is uniformly positive.

The organization of this paper is as follows.
In Section \ref{sec:complex}, we prove the main theorem; in Section \ref{sec:schur}, we give a sufficient condition for preservation; and in Section \ref{sec:real}, we give a real counterpart.

\subsection*{Acknowledgment}\label{subsec-ack}
The author is supported by Grant-in-Aid for Early-Career Scientists $\sharp$23K12978 from the Japan Society for the Promotion of Science (JSPS).

\section{A complex counterexample}\label{sec:complex}

Before proving Theorem \ref{mainthm:counterexample}, we introduce some notation.
Write $\tau=t+\ai s\in\C^r_{\{\tau_1,\ldots,\tau_r\}}$ and $z=x+\ai y\in\C^n_{\{z_1,\ldots,z_n\}}$, and let $\Delta\subset\C$ be the unit disc.
For a Hermitian matrix $H$ with ordered eigenvalues
$\lambda_1(H)\leq\cdots\leq\lambda_N(H)$, set
\begin{equation}\label{eq:ky-fan}
	\Gamma_q(H):=\lambda_1(H)+\cdots+\lambda_q(H).
\end{equation}
For $K\geq1$, put
\[
	U=\Delta^r,\quad D=\Delta^n,\quad \pi:U\times D\longrightarrow U,\quad \Omega_\tau=\pi^{-1}(\tau)=\{\tau\}\times D,
\]
\[
	M_K=K+\frac15,\qquad
	d_K(\tau_1)=1+2Kt_1+4K^4|\tau_1|^4,
\]
where $\pi$ is the projection defined by $(\tau, z)\mapsto \tau$, and define
\begin{align}\label{eq:complex-weight}
	\varphi_K(\tau,z)
	={} & M_K\sum_{j=2}^{q}|\tau_j|^2
	+2M_K\sum_{j=q+1}^{r}|\tau_j|^2\notag \\
	    & +d_K(\tau_1)|z_1|^2
	+2M_K\sum_{\ell=2}^{n}|z_\ell|^2.
\end{align}
Set
\[
	f_{\C}(a):=\int_0^1e^{-a\sigma}\,\mathrm{d}\sigma,
	\qquad
	\gamma_{\C}:=(\log f_{\C})''(1).
\]

Since $f_{\C}(a)=(1-e^{-a})/a$ for $a>0$, direct computation gives
\begin{equation}\label{eq:gamma-complex}
	\gamma_{\C}=1-\frac{e}{(e-1)^2}
	=1-\frac1{e+e^{-1}-2}.
\end{equation}
On the other hand,
\[
	e+e^{-1}-2=2\left(\frac1{2!}+\frac1{4!}+\cdots\right)>1.
\]
Hence $\gamma_{\C}>0$.

\begin{proof}[Proof of Theorem \ref{mainthm:counterexample}]
	First we show that $d_K \geq 1/4$.
	Put $K|\tau_1|=\eta$.
	Considering that $-\eta \leq Kt_1 \leq \eta$,
	we have
	$d_K = 1 + 2Kt_1+4\eta^4\geq 1-2\eta+4\eta^4$.
	We can easily check that $1-2\eta+4\eta^4 \geq 1/4$.

	The only exceptional block of $\Levi_{\varphi_K}$ is the $(\tau_1,z_1)$ block
	\[
		A_K=
		\begin{pmatrix}
			16K^4|\tau_1|^2|z_1|^2                           &
			(K+8K^4\tau_1\overline\tau_1^{\,2})z_1             \\
			(K+8K^4\overline\tau_1\tau_1^{\,2})\overline z_1 &
			d_K
		\end{pmatrix}.
	\]
	Write its eigenvalues as $\lambda_-\leq\lambda_+$.
	Put $\sigma=|z_1|^2$.
	The determinant of $A_K+KI_2$ is affine in $\sigma$, and its values at $\sigma=0$ and $\sigma=1$ are
	\[
		K(K+d_K)>0
	\]
	and
	\[
		K\left(d_K+16K\eta^2(K+1+Kt_1)\right)>0,
	\]
	respectively.
	Then it follows that $\det(A_K+KI_2)>0$ for $0\leq\sigma\leq1$.
	We also have that $\tr(A_K+KI_2)>0$,
	which means that
	\[
		A_K+KI_2\succ0.
	\]
	Thus $\lambda_->-K$.

	The remaining eigenvalues are $M_K$ with multiplicity $q-1$ and $2M_K$ with multiplicity $r-q+n-1$.
	If $\lambda_+\leq M_K$, then
	\[
		\Gamma_q(\Levi_{\varphi_K})
		=\tr A_K+(q-2)M_K
		\geq d_K>\frac15.
	\]
	If $\lambda_+>M_K$, then
	\[
		\Gamma_q(\Levi_{\varphi_K})
		\geq \min\{\lambda_-,M_K\}+(q-1)M_K
		> -K+(q-1)M_K
		\geq\frac15.
	\]

	Finally, we can see that
	\begin{align*}
		\widetilde\varphi_K(\tau)
		={} & M_K\sum_{j=2}^{q}|\tau_j|^2
		+2M_K\sum_{j=q+1}^{r}|\tau_j|^2               \\
		    & -\log f_{\C}(d_K(\tau_1))+\text{const}.
	\end{align*}
	At the origin,
	\[
		d_K(0)=1,\qquad
		(d_K)_{\tau_1}(0)=K,\qquad
		(d_K)_{\tau_1\overline\tau_1}(0)=0.
	\]
	The chain rule gives
	\begin{align*}
		\frac{\partial^2}{\partial \tau_1\partial\overline \tau_1}
		\bigl[-\log f_{\C}(d_K(\tau_1))\bigr]
		 & =
		-(\log f_{\C})''(d_K)_{\tau_1}(d_K)_{\overline \tau_1}
		-(\log f_{\C})'(d_K)_{\tau_1\overline \tau_1}.
	\end{align*}
	Note that at $\tau=0$, $(\log f_{\C})''(1) = \gamma_{\C}$.
	Therefore the eigenvalues of $\Levi_{\widetilde\varphi_K}(0)$ are
	\[
		-K^2\gamma_{\C},\quad
		\underbrace{M_K,\ldots,M_K}_{q-1},\quad
		\underbrace{2M_K,\ldots,2M_K}_{r-q}.
	\]
	Hence
	\[
		\Gamma_q(\Levi_{\widetilde\varphi_K}(0))
		=(q-1)\left(K+\frac15\right)-K^2\gamma_{\C}
		\longrightarrow-\infty.
	\]
\end{proof}

We present an example with $r=q=2$ and $n=1$ to illustrate the above proof.
\begin{example}\label{ex:complex-two-base}
	Take $r=q=2$ and $n=1$.
	On $\Delta^2\times\Delta$, set
	\[
		\varphi_K(\tau,z)
		=\left(K+\frac15\right)|\tau_2|^2
		+\left(1+2Kt_1+4K^4|\tau_1|^4\right)|z|^2.
	\]
	The sum of the two smallest eigenvalues of $\Levi_{\varphi_K}$ is greater than $1/5$, whereas
	\[
		\Levi_{\widetilde\varphi_K}(0)
		=\begin{pmatrix}
			-K^2\left(1-\dfrac{e}{(e-1)^2}\right) & 0          \\
			0                                     & K+\dfrac15
		\end{pmatrix}.
	\]
	The sum of these two eigenvalues
	\[
		K+\frac15-K^2\left(1-\frac{e}{(e-1)^2}\right)
	\]
	tends to $-\infty$ (and is already negative for $K=100$).
\end{example}

The following example shows that ordinary $q$-positivity is not
preserved by marginalization either.

\begin{example}\label{ex:ordinary-one-positive}
	Let $U=\Delta_\varepsilon^2\subset\C^2$, where $\varepsilon>0$
	is sufficiently small, and let $D=\Delta\subset\C$.
	Set
	\[
		\varphi(\tau,z)
		=|z|^2+(7|z|^2-3)|\tau_1|^2
		+(1-4|z|^4)|\tau_2|^2.
	\]
	This function is smooth and independent of $\arg z$.
	In the coordinates $(\tau_1,\tau_2,z)$, its Levi matrix is
	\[
		\Levi_\varphi(\tau,z)
		=
		\begin{pmatrix}
			7|z|^2-3
			 & 0
			 & 7\overline{\tau_1}z             \\[3pt]
			0
			 & 1-4|z|^4
			 & -8|z|^2\overline{\tau_2}z       \\[3pt]
			7\tau_1\overline z
			 & -8|z|^2\tau_2\overline z
			 & 1+7|\tau_1|^2-16|z|^2|\tau_2|^2
		\end{pmatrix}.
	\]
	In particular,
	\[
		\Levi_\varphi(0,z)
		=\operatorname{diag}(7|z|^2-3,\,1-4|z|^4,\,1).
	\]
	The first diagonal entry is positive for $|z|^2>3/7$,
	while the second is positive for $|z|^2<1/2$.
	Since $3/7<1/2$, this matrix has at least two positive
	eigenvalues for every $z\in\overline\Delta$.
	By continuity and the compactness of $\overline\Delta$,
	we may choose $\varepsilon$ sufficiently small so that
	$\varphi$ is $1$-positive and
	$\varphi_{z\overline z}\geq1/2$ on
	$U\times\overline\Delta$.

	Direct computation shows that
	\[
		\widetilde{\varphi}_{\tau_j\bar{\tau}_k}(\tau) = \frac{\left(\int_D (\varphi_{\tau_j\bar{\tau}_k}-\varphi_{\tau_j}\varphi_{\bar{\tau}_k})e^{-\varphi(\tau,z)} \right)\int_D e^{-\varphi(\tau,z)} + \int_D \varphi_{\tau_j}e^{-\varphi(\tau,z)}\int_D \varphi_{\bar{\tau}_k}e^{-\varphi(\tau,z)}}{\left(\int_D e^{-\varphi(\tau,z)}\right)^2}.
	\]
	Since $\varphi_{\tau_j}(0,z)=0$ for $j=1,2$, we obtain
	\begin{align*}
		\Levi_{\widetilde\varphi}(0)
		 & =
		\frac{\displaystyle\int_D
			\operatorname{diag}(7|z|^2-3,\,1-4|z|^4)
			e^{-|z|^2}\,\mathrm{d}V(z)}
		{\displaystyle\int_D
		e^{-|z|^2}\,\mathrm{d}V(z)}                                           \\
		 & =
		\frac{\displaystyle\int_0^{2\pi}\int_0^1
			\operatorname{diag}(7r^2-3,\,1-4r^4)
			e^{-r^2}\,r\,\mathrm{d}r\,\mathrm{d}\theta}
		{\displaystyle\int_0^{2\pi}\int_0^1
		e^{-r^2}\,r\,\mathrm{d}r\,\mathrm{d}\theta} \qquad (z=re^{\ai\theta}) \\
		 & =
		\frac{\displaystyle\int_0^1
			\operatorname{diag}(7\sigma-3,\,1-4\sigma^2)
			e^{-\sigma}\,\mathrm{d}\sigma}
		{\displaystyle\int_0^1
		e^{-\sigma}\,\mathrm{d}\sigma} \qquad (\sigma=r^2)                    \\
		 & =
		\frac{1}{e-1}
		\begin{pmatrix}
			4e-11 & 0     \\
			0     & 19-7e
		\end{pmatrix}.
	\end{align*}
	Both diagonal entries are negative, since
	$19/7<e<11/4$.
	Thus $\Levi_{\widetilde\varphi}(0)$ is negative definite,
	and $\widetilde\varphi$ is not $1$-positive at the origin.
\end{example}

\section{The horizontal Schur complement}\label{sec:schur}

In this section, we study a sufficient condition for the preservation of partial positivity under marginalization.
Let $U\subset\C^r$ be a domain, $D\Subset\C^n$ be a pseudoconvex Reinhardt domain and $\varphi$ be a smooth real-valued function on $U\times \overline{D}$.
Write the total Levi matrix in base--fiber blocks as
\begin{equation}\label{eq:block-levi}
	\Levi_\varphi=
	\begin{pmatrix}
		A   & B \\
		B^* & C
	\end{pmatrix},
\end{equation}
where $A=(\varphi_{\tau_j\overline\tau_k})$, $B=(\varphi_{\tau_j\overline z_\beta})$ and
$C=(\varphi_{z_\alpha\overline z_\beta})$.
We assume that $C\succ0$ on $U\times\overline D$ and set
\begin{equation}\label{eq:schur}
	S_\varphi:=A-BC^{-1}B^*.
\end{equation}
Following the notation of Berndtsson's paper \cite{Ber09}, we let $A^2_\tau$ be the weighted Bergman space of holomorphic functions on $D$ with the norm
\[
	\left\|h\right\|^2_\tau = \int_D |h|^2 e^{-\varphi(\tau,z)}\,\dV(z).
\]
We consider the vector bundle $F$ over $U$ with fiber $F_\tau = L^2(D, e^{-\varphi(\tau,\cdot)})$ and the vector bundle $E$ over $U$ with fiber $E_\tau=A^2_\tau \subset F_\tau$.
Berndtsson's curvature calculation \cite[(3.1)]{Ber09} shows that for $u_1,\ldots,u_r\in E_\tau$, we have
\[
	\sum_{j,k=1}^{r}
	\left\langle\Theta^E_{j\bar k}u_j,u_k\right\rangle_\tau
	\ge
	\int_D
	\sum_{j,k=1}^{r}
	(S_\varphi)_{jk}(\tau,z)\,
	u_j(z)\overline{u_k(z)}
	e^{-\varphi(\tau,z)}\,dV(z).
\]
Note that Berndtsson assumes that $\varphi$ is plurisubharmonic, but the above inequality holds under the weaker assumption $C\succ0$ since the proof relies only on the fiberwise H\"ormander $L^2$-estimate.
Using this inequality, we can prove the following theorem.

\begin{theorem}\label{thm:berndtsson-ky-fan}
	Let $\varphi$ be independent of $\arg(z_j)$ for $j=1,\ldots,n$, and assume $C\succ0$.
	Then
	\begin{equation}\label{eq:berndtsson-estimate}
		\Levi_{\widetilde\varphi}(\tau)
		\succeq
		\frac{\displaystyle\int_D
		S_\varphi(\tau,z)e^{-\varphi(\tau,z)}\,\dV(z)}
		{\displaystyle\int_D e^{-\varphi(\tau,z)}\,\dV(z)}.
	\end{equation}
	Consequently, if $1\leq q\leq r$ and
	\[
		\Gamma_q(S_\varphi(\tau,z))\geq c
	\]
	for $c\in \R$ on $U\times D$, then
	\[
		\Gamma_q(\Levi_{\widetilde\varphi}(\tau))\geq c
	\]
	on $U$.
\end{theorem}

This theorem shows that if we assume that the Schur complement $S_\varphi$ is uniformly $(q-1)$-positive, then $\widetilde\varphi$ is also uniformly $(q-1)$-positive.

\begin{proof}[Proof of Theorem \ref{thm:berndtsson-ky-fan}]
	From \cite{Ber09}, we can see that $D^E = P_\tau D^F$ and $D^F_{\tau_j} = \partial_{\tau_j}-\varphi_{\tau_j}$, where $P_\tau:F_\tau\to E_\tau$ is the orthogonal projection.
	Hence, it holds that
	\[
		D^E 1 = - \sum_j P_\tau(\varphi_{\tau_j})d\tau_j.
	\]
	Let $\theta = (\theta_1,\ldots,\theta_n)\in \mathbf{T}^n$ and define $T_\theta:F_\tau\to F_\tau$ by
	\[
		(T_\theta f)(z)=f(R_\theta z),
	\]
	where $R_\theta z = (e^{\ai \theta_1}z_1,\ldots,e^{\ai \theta_n}z_n)$.
	This operator preserves $E_\tau$.
	It also preserves $E_\tau^\perp$ since if $f\in E^\perp_\tau$ and $g\in E_\tau$, then
	\[
		\langle T_\theta f, g\rangle_\tau = \langle f, T_{-\theta} g\rangle_\tau = 0.
	\]
	Apply $T_\theta$ to the unique decomposition $f=P_\tau f + (I-P_\tau)f$ for $f\in F_\tau$.
	The first transformed term is in $E_\tau$ and the second transformed term is in $E_\tau^\perp$.
	Uniqueness of the decomposition implies that $T_\theta P_\tau = P_\tau T_\theta$.
	Combining the above facts, we have that $P_\tau(\varphi_{\tau_j})$ is a torus-invariant holomorphic function on $D$, and hence is constant in $z$ (note that $\varphi_{\tau_j}$ is also torus-invariant).
	Therefore $D^E 1$ takes values in $L_\tau:=\C 1$.
	By the Leibniz rule, the same holds for $D^E s$ for every smooth section $s$ of $L$, which implies that the second fundamental form of $L$ vanishes.
	Hence, if we take $u_j=\xi_j \cdot 1$ for $\xi=(\xi_1,\ldots,\xi_r) \in \C^r$, we obtain
	\[
		\sum_{j,k=1}^{r}
		\left\langle\Theta^E_{j\bar k}\xi_j \cdot 1,\xi_k \cdot 1\right\rangle_\tau
		=
		\Theta^L(\xi, \xi) \left\| 1\right\|^2_\tau
		\ge
		\int_D
		S_\varphi(\tau,z)(\xi, \xi)
		e^{-\varphi(\tau,z)}\,dV(z),
	\]
	equivalently,
	\[
		\Levi_{\widetilde\varphi}(\tau)(\xi,\xi) \geq \frac{\displaystyle\int_D S_\varphi(\tau,z)(\xi, \xi)e^{-\varphi(\tau,z)}\,dV(z)}{\displaystyle\int_D e^{-\varphi(\tau,z)}\,dV(z)}.
	\]

	The Ky Fan minimum principle \cite{Fan49} implies that
	\[
		\Gamma_q(H)
		=\min_{\substack{P=P^*=P^2\\ \operatorname{rank}P=q}}\tr(PH).
	\]
	This gives the following two properties.
	First, $\Gamma_q$ is concave: for $0\leq \theta\leq 1$,
	\begin{align*}
		\Gamma_q(\theta H+(1-\theta)K)
		 & =\min_P\bigl(\theta\tr(PH)+(1-\theta)\tr(PK)\bigr) \\
		 & \ge\theta\Gamma_q(H)+(1-\theta)\Gamma_q(K).
	\end{align*}
	Second, let $R\succeq0$ and $P$ be an orthogonal projection.
	Then $\tr(PRP) = \tr(PR)$ and $\tr(PRP)\geq 0$.
	Therefore,
	\begin{align*}
		\Gamma_q(H+R)
		 & =\min_P\bigl(\tr(P(H+R))\bigr)       \\
		 & \geq \min_P \tr(PH) + \min_P \tr(PR) \\
		 & \geq \Gamma_q(H).
	\end{align*}
	Let
	\[
		\overline{S}(\tau) = \frac{\displaystyle\int_D
		S_\varphi(\tau,z)e^{-\varphi(\tau,z)}\,\dV(z)}
		{\displaystyle\int_D e^{-\varphi(\tau,z)}\,\dV(z)}.
	\]
	From what we have discussed, we obtain
	\begin{align*}
		\Gamma_q(\Levi_{\widetilde\varphi}(\tau))
		 & \ge \Gamma_q(\overline S(\tau))                 \\
		 & \ge \frac{\displaystyle\int_D
		\Gamma_q(S_\varphi(\tau,z))e^{-\varphi(\tau,z)}\,\dV(z)}
		{\displaystyle\int_D e^{-\varphi(\tau,z)}\,\dV(z)} \\
		 & \ge c.\qedhere
	\end{align*}
\end{proof}

Let us show an example to which Theorem \ref{thm:berndtsson-ky-fan} applies.

\begin{example}\label{ex:mixed-schur}
	Let $U=\Delta^2$, $D=\Delta$ and $q=2$, and set
	\[
		\varphi(\tau,z)
		=-3|\tau_1|^2+4|\tau_2|^2+(1+|\tau_1|^2)|z|^2.
	\]
	We can see that
	\[
		\Levi_\varphi(\tau,z)=
		\begin{pmatrix}
			-3+|z|^2          & 0 & \overline\tau_1z \\
			0                 & 4 & 0                \\
			\tau_1\overline z & 0 & 1+|\tau_1|^2
		\end{pmatrix}.
	\]
	Applying the Ky Fan minimum principle to the $(\tau_1,z)$ coordinate plane, we obtain
	\begin{equation*}
		\Gamma_2(\Levi_\varphi(\tau,z)) \leq -2+|\tau_1|^2+|z|^2<0
	\end{equation*}
	for every $(\tau,z)\in U\times D$.
	However, we have
	\[
		C=1+|\tau_1|^2,\qquad
		B=\begin{pmatrix}\overline\tau_1z\\0\end{pmatrix},\qquad
		S_\varphi=
		\begin{pmatrix}
			-3+\dfrac{|z|^2}{1+|\tau_1|^2} & 0 \\
			0                              & 4
		\end{pmatrix},
	\]
	and
	\[
		\Gamma_2(S_\varphi)=1+\frac{|z|^2}{1+|\tau_1|^2}\geq1,
	\]
	for every $(\tau,z)\in U\times D$.
	Thus Theorem \ref{thm:berndtsson-ky-fan} gives $\Gamma_2(\Levi_{\widetilde\varphi}(\tau))\geq1$ for every $\tau\in U$.
\end{example}

As Theorem \ref{mainthm:counterexample} shows, the assumption $\Gamma_q(\Levi_\varphi)\geq c$ does not imply $\Gamma_q(\Levi_{\widetilde\varphi})\geq c$.
However, if we impose additional assumptions, we can obtain the positivity of $\Gamma_q(\Levi_{\widetilde\varphi})$ from that of $\Gamma_q(\Levi_\varphi)$.

\begin{corollary}\label{cor:mixed-block}
	In the setting of Theorem \ref{thm:berndtsson-ky-fan}, let $1\leq q\leq r$, $c\in\R$, $\delta>0$ and $M\geq0$.
	Assume on $U\times D$ that
	\[
		\Gamma_q(\Levi_\varphi)\geq c,\qquad
		C\succeq\delta I_n,\qquad
		\lVert B\rVert_{\mathrm{op}}\leq M.
	\]
	Then, on $U$,
	\[
		\Gamma_q(\Levi_{\widetilde\varphi})
		\geq c-\frac{qM^2}{\delta}.
	\]
\end{corollary}

For example, if $c>0$ and $qM^2<c\delta$, then $\Gamma_q(\Levi_{\widetilde\varphi})$ is positive.

\begin{proof}
	Let $\lambda_1\leq\cdots\leq\lambda_{r+n}$ be the eigenvalues of
	$\Levi_\varphi$, and let $\mu_1\leq\cdots\leq\mu_r$ be those of $A$.
	Since $A$ is a principal submatrix of $\Levi_\varphi$, the eigenvalue
	interlacing theorem \cite[Theorem~4.3.28]{HJ13} gives $\lambda_j\leq\mu_j$
	for $1\leq j\leq r$. Hence
	\[
		\Gamma_q(A)=\sum_{j=1}^q\mu_j
		\geq\sum_{j=1}^q\lambda_j
		=\Gamma_q(\Levi_\varphi)\geq c.
	\]
	Moreover, we have from the assumptions that
	\[
		0\prec C^{-1}\preceq\frac1\delta I_n,
	\]
	and for $\xi\in \C^r$
	\[
		\langle (BC^{-1}B^*)\xi, \xi\rangle = \langle C^{-1}B^*\xi, B^*\xi\rangle \leq \frac{1}{\delta}\left\| B^*\xi\right\|^2 \leq \frac{M^2}{\delta}\|\xi\|^2,
	\]
	which implies that $S_\varphi=A-BC^{-1}B^* \succeq A - \frac{M^2}{\delta}I_r$.
	Thus it follows that
	\[
		\Gamma_q(S_\varphi) \geq \Gamma_q\left(A-\frac{M^2}{\delta}I_r \right)=\Gamma_q(A)-\frac{qM^2}{\delta}\geq c-\frac{qM^2}{\delta}.
	\]
	By Theorem \ref{thm:berndtsson-ky-fan}, we obtain the desired estimate.
\end{proof}

\begin{example}\label{ex:large-fiber-small-mixed}
	Let $U=\Delta^2$, $D=\Delta$ and $q=2$, and set
	\[
		\varphi(\tau,z)
		=-|\tau_1|^2+2|\tau_2|^2
		+\left(10+\frac{|\tau_1|^2}{10}\right)|z|^2.
	\]
	We have
	\begin{align*}
		\Levi_\varphi
		 & =\begin{pmatrix}
			    -1+\dfrac{|z|^2}{10}          & 0 & \dfrac{\overline\tau_1z}{10} \\
			    0                             & 2 & 0                            \\
			    \dfrac{\tau_1\overline z}{10} & 0 & 10+\dfrac{|\tau_1|^2}{10}
		    \end{pmatrix} \\
		 & =\operatorname{diag}(-1,2,10)
		+\frac1{10}
		\begin{pmatrix}z\\0\\\tau_1\end{pmatrix}
		\begin{pmatrix}\overline z&0&\overline\tau_1\end{pmatrix}
		\succeq\operatorname{diag}(-1,2,10).
	\end{align*}
	Consequently, $\Gamma_2(\Levi_\varphi)\geq-1+2=1$.
	Moreover,
	\[
		C=10+\frac{|\tau_1|^2}{10}\geq10,\qquad
		B=\begin{pmatrix}\overline\tau_1z/10\\0\end{pmatrix},\qquad
		\|B\|_{\mathrm{op}}=\frac{|\tau_1||z|}{10}\leq\frac1{10}.
	\]
	Taking $c=1$, $\delta=10$ and $M=1/10$, Corollary \ref{cor:mixed-block} gives
	\[
		\Gamma_2(\Levi_{\widetilde\varphi})
		\geq1-\frac{2(1/10)^2}{10}
		=\frac{499}{500}>0
	\]
	on $U$.
\end{example}

\section{A real counterpart}\label{sec:real}

In this section, we show a real counterpart of Theorem \ref{mainthm:counterexample}.
Before proving the result, we introduce some notation and preliminary results.
For a real symmetric matrix, we use the same notation $\Gamma_q$ as in \eqref{eq:ky-fan}.
For a smooth function $\psi$ on $I\times X\subset\R^r_{\{t_1,\ldots,t_r\}}\times\R^n_{\{x_1,\ldots,x_n\}}$, set
\[
	\widetilde\psi(t):=-\log\int_Xe^{-\psi(t,x)}\,\dx.
\]
Define
\[
	f_{\R}(a):=\int_{-1}^1e^{-au^2/2}\,\mathrm{d}u,
	\qquad
	\gamma_{\R}:=(\log f_{\R})''(1).
\]
Integration by parts gives
\[
	af_{\R}'(a)=e^{-a/2}-\frac12f_{\R}(a).
\]
Differentiating this identity and evaluating at $a=1$, we obtain
\begin{equation}\label{eq:real-log-convexity}
	\gamma_{\R}=\frac12-\frac{e^{-1/2}}{f_{\R}(1)}
	-\frac{e^{-1}}{f_{\R}(1)^2}.
\end{equation}
On the other hand, $e^v>1+v+v^2/2$ for $v>0$ gives
\begin{align*}
	f_{\R}(1)
	 & =2e^{-1/2}\int_0^1e^{(1-u^2)/2}\,\mathrm{d}u                                      \\
	 & >2e^{-1/2}\int_0^1\left(1+\frac{1-u^2}{2}+\frac{(1-u^2)^2}{8}\right)\,\mathrm{d}u \\
	 & =2e^{-1/2}\left(1+\frac13+\frac1{15}\right)
	=\frac{14}{5}e^{-1/2}.
\end{align*}
Hence $\gamma_{\R}>\frac12-\frac5{14}-\frac{25}{196}=\frac3{196}>0$.

Choose a smooth function $\rho:\R\to\R$ such that
\begin{equation}\label{eq:rho-properties}
	\rho\geq\frac34,\qquad 0\leq\rho'\leq1,\qquad\rho''\geq0,
\end{equation}
and $\rho(u)=1+u$ in a neighborhood of $0$.

\begin{theorem}\label{thm:real-counterexample}
	Fix $r\geq q\geq2$ and $n\geq1$.
	Put
	\[
		I=(-1,1)^r,\qquad X=(-1,1)^n,
	\]
	and, for $K\geq1$, set
	\[
		M_K=K+\frac15,\qquad
		d_K(t_1)=\rho(Kt_1).
	\]
	Define
	\begin{align}\label{eq:real-weight}
		\psi_K(t,x)
		={} & \frac{M_K}{2}\sum_{j=2}^{q}t_j^2
		+M_K\sum_{j=q+1}^{r}t_j^2\notag        \\
		    & +\frac12d_K(t_1)x_1^2
		+M_K\sum_{\ell=2}^{n}x_\ell^2.
	\end{align}
	Then
	\[
		\Gamma_q(D^2\psi_K)>\frac15
	\]
	on $I\times X$.
	Nevertheless,
	\begin{equation}\label{eq:real-defect}
		\Gamma_q(D^2\widetilde\psi_K(0))
		=(q-1)\left(K+\frac15\right)-K^2\gamma_{\R}
		\longrightarrow-\infty.
	\end{equation}
\end{theorem}

\begin{proof}[Proof of Theorem \ref{thm:real-counterexample}]
	Put $\eta=Kt_1$.
	The exceptional $(t_1,x_1)$ block of $D^2\psi_K$ is
	\[
		A_K=
		\begin{pmatrix}
			\frac12K^2\rho''(\eta)x_1^2 & K\rho'(\eta)x_1 \\
			K\rho'(\eta)x_1             & \rho(\eta)
		\end{pmatrix}.
	\]
	Write its eigenvalues as $\lambda_-\leq\lambda_+$.
	Since $0\leq\rho'\leq1$ and $\rho''\geq0$,
	\begin{align*}
		\det(A_K+KI_2)
		 & =\left(K+\frac12K^2\rho''(\eta)x_1^2\right)(K+\rho(\eta))
		-K^2\rho'(\eta)^2x_1^2                                       \\
		 & \geq K^2 + K\rho(\eta) - K^2x_1^2>0.
	\end{align*}
	It is also true that $\tr(A_K + KI_2)>0$, which means that $A_K + KI_2\succ0$.
	Thus $\lambda_->-K$, while $\tr A_K\geq\rho(\eta)\geq3/4$.
	The remaining Hessian eigenvalues are $M_K$ with multiplicity $q-1$ and $2M_K$ with multiplicity $r-q+n-1$.
	The same two-case argument as in the proof of Theorem \ref{mainthm:counterexample} gives
	\[
		\Gamma_q(D^2\psi_K)>\frac15.
	\]

	Integration in the fiber gives
	\[
		\widetilde\psi_K(t)
		=\frac{M_K}{2}\sum_{j=2}^q t_j^2
		+M_K\sum_{j=q+1}^r t_j^2
		-\log f_{\R}(d_K(t_1))+\text{const}.
	\]
	Since $\rho(u)=1+u$ near $0$, we have
	$d_K(0)=1$, $d_K'(0)=K$ and $d_K''(0)=0$.
	The same argument as in the proof of Theorem \ref{mainthm:counterexample} gives the Hessian eigenvalues
	\[
		-K^2\gamma_{\R},\quad
		\underbrace{M_K,\ldots,M_K}_{q-1},\quad
		\underbrace{2M_K,\ldots,2M_K}_{r-q}
	\]
	at the origin.
	This proves \eqref{eq:real-defect}.
\end{proof}

\subsection*{Use of AI}
This work was developed through iterative mathematical discussions between the author and OpenAI's Codex (GPT-5.6 Sol and GPT-6 Astra).
The author contributed ideas and mathematical arguments to this process.
The explicit counterexample constructions in Sections 2 and 4 were proposed by the AI rather than devised by the author.
The author independently verified all mathematical statements and proofs and takes full responsibility for the entire content of the paper.


\end{document}